%% file: CialdeaLeonessa.tex
\documentclass[graybox]{SNmult}

\usepackage{type1cm}        
\usepackage{makeidx}         
\usepackage{graphicx}        
\usepackage{multicol}        
\usepackage[bottom]{footmisc}

\usepackage{newtxtext}       %
\usepackage[varvw]{newtxmath}       

\makeindex             

\begin{document}
\title*{On the Robin  problem
for the Laplace equation in multiply connected domains}
\author{Alberto Cialdea\orcidID{0000-0002-0009-5957} and Vita Leonessa\orcidID{0000-0001-9547-8397}}
\institute{Alberto Cialdea \at  Università degli Studi della Basilicata, Via dell'Ateneo Lucano n. 10, 85100 Potenza (Italy), \email{alberto.cialdea@unibas.it}
\and Vita Leonessa \at Università degli Studi della Basilicata,  Via dell'Ateneo Lucano n. 10, 85100 Potenza (Italy), \email{vita.leonessa@unibas.it}}
%
%
\maketitle
\abstract*{
This paper complements the existing theory developed in \cite{CiLeMa2012} for the Dirichlet and Neumann  problems for the Laplace equation, in multiply connected domains. Within the framework of layer potential methods, we study the Laplace equation  under Robin boundary conditions, representing the solutions by means of a double layer potential. We observe that   the classical approach  searches the solutions in terms of a single layer potential.
}

\abstract{This paper complements the existing theory developed in \cite{CiLeMa2012} for the Dirichlet and Neumann  problems for the Laplace equation, in multiply connected domains. Within the framework of layer potential methods, we study the Laplace equation  under Robin boundary conditions, representing the solutions by means of a double layer potential. We observe that   the classical approach  searches the solutions in terms of a single layer potential.
}

\section{Introduction}\label{sec:intro}

Potential theory provides a powerful framework for the analysis of boundary value problems for partial differential equations, allowing solutions to be represented through layer potentials. 
This gave rise to interest in developing methods of potential theory, which in turn led to the exploration of new theoretical frameworks and the development of novel techniques.

 In earlier work \cite{cialdea}, the first author introduced a method for determining the solution of a boundary integral equation of the first kind, arising from the imposition of Dirichlet boundary condition for the Laplace equation, through a single layer potential representation. Note that usually the Dirichlet problem is solved by means of a double layer potential. 
A closely related problem concerns the construction of solutions to the Neumann problem using a double layer potential rather than the more standard single layer representation (see, for instance, \cite[pp. 28–29]{CialdeaHsiao1995} in the case of simply connected domains).
 
The  approach proposed  in \cite{cialdea}  is based on the theories of singular
integral operators and  of reducible operators,  and may be regarded as a higher dimensional generalization of Muskhelishvili’s method (see \cite{ACCialdeanew}). This framework was later developed further and applied to a variety of boundary value problems for different partial differential equations and systems, both in simply and multiply connected domains (see, e.g.  \cite{Cialema2011,ACCiLeMa4,cialema2023} and the references therein).
 
The aim of the present paper is to complement the above mentioned investigation by showing that the same approach can also be applied to boundary value problems with Robin boundary conditions. More precisely, we prove that the corresponding Laplace problem in multiply connected domains $\Omega \subset \mathbb{R}^n$ ($n\geq 2$) 
 can be solved by means of a double layer potential. As in the previously considered cases, the resulting representation of the solution differs from the classical one which is based on a single layer potential  (see, e.g. \cite{medkova} and the references therein). Specifically, the problem consists in finding a  function $u$ in   $\Omega$ such that
\begin{equation}\label{Robin P}
\Delta u = 0 \quad \text{in } \Omega, \qquad
  \frac{\partial u}{\partial \nu}+hu = g 
\quad \text{on } \Sigma=\partial\Omega,
\end{equation}
where $h$ and $g$  are given real valued functions, satisfying suitable hypotheses, and $\partial u/\partial \nu$ denotes
the outward normal derivative.  Note that, in the particular case $h= 0$,   problem  \eqref{Robin P} reduces to  the Neumann problem.

The paper is organized as follows. 
After summarizing notations and definitions in Section \ref{sec:Notation},
in Section \ref{sec:DirichletNeumann}  we collect the main results contained in \cite{CiLeMa2012}. 
Finally, in Section \ref{sec:Robin}  we give  an existence (and uniqueness) theorem   for the solution of
the Robin problem with datum in $L^p(\Sigma)$  in the form of a double layer
potential with density in $W^{1, p}(\Sigma)$.

\section{Notations and preliminaries}\label{sec:Notation}

Let $0\leq k\leq n$ and  let $T\subseteq\mathbb{R}^n$ ($n\geq 2$) be a domain (that is, an open connected set).   
A differential form of degree 
$k$   (briefly a $k$-form) on   $T $  is a function defined on
$T$ taking  values  in the space of $k$-covectors on $\mathbb{R}^n$. If $(x_1,\ldots,x_n)$ is an admissible coordinate
system, then a $k$-form $u$ can be written  as
\begin{equation*}
      u=\frac{1}{k!}  u_{s_{1}\ldots s_{k}}\mathrm{d}x^{s_{1}}  \cdots   \mathrm{d}x^{s_{k}},
\end{equation*}
 where $u_{s_1\ldots s_k}$ are the components  of a skew-symmetric covariant
tensor.

With the symbol $C_k^{h}(T)$ we denote the space of all $k$-forms on $T$ whose components
are $h$-times continuously differentiable   in a
coordinate system of class $C^{h+1}$ (and therefore in every
coordinate system of class $C^{h+1}$).

Similarly,  $L_k^p(T)$ stands for  the
space of all $k$-forms on $T$ whose components are  real-valued $L^p$  functions in a
coordinate system of class $C^1$ (and therefore in every coordinate
system of class $C^1$).

 Now, let $u\in C_{k}^{1}(T)$. The
   differential   $\mathrm{d}u$ of $u$ is the   $(k+1)$-form defined by
 \begin{equation*}
       \mathrm{d}u=\frac{1}{k!} \frac{\partial}{\partial x^{j}}u_{s_{1}\ldots
       s_{k}}\mathrm{d}x^{j} \mathrm{d}x^{s_{1}}  \cdots \mathrm{d}x^{s_{k}}.
 \end{equation*}

 Moreover, if $u\in C_{k}^0(T)$, then  the  Hodge  adjoint (Hodge star)   of $u$ is the  
 ${(n-k)}$-form  
 \begin{equation*}
       \ast u=\frac{1}{k!(n-k)!} \delta^{1\ldots \ldots \ldots \ldots .. n}_
              {s_{1}\ldots s_{k}i_{1}\ldots i_{n-k}}
              u_{s_{1}\ldots s_{k}}\mathrm{d}x^{i_{1}} \ldots  \mathrm{d}x^{i_{n-k}}\,,
 \end{equation*}
 where $ \delta^{1\ldots \ldots \ldots \ldots .. n}_
              {s_{1}\ldots s_{k}i_{1}\ldots i_{n-k}}$ is the Kronecker symbol
 Observe  that   
 \[
 \ast\! \ast u=(-1)^{k(n-k)}u.
 \]
 Using the Hodge star operator, 
 if $u\in C_k^{1}(T)$ one defines the co-differential of $u$ as  the  ${(k-1)}$-form
 \begin{equation*}
       \d u=(-1)^{n(k+1)+1}\ast \mathrm{d}\ast u\,.
 \end{equation*}
 
In what follows, we use a particular double $k$-form introduced by Hodge
in \cite{hodge}, denoted by  $s_{k}(x,y)$  and defined by
\begin{equation*}\label{lab:1FormaHodge}
s_{k}(x,y)=\sum_{j_1<\ldots <j_{k}}s(x,y)\mathrm{d}x^{j_1}\cdots
\mathrm{d}x^{j_{k}} \mathrm{d}y^{j_1}\cdots  \mathrm{d}y^{j_{k}},
\end{equation*}
where $s(x,y)$  is
the fundamental solution of the
Laplace operator
\begin{eqnarray*}
      s(x,y)= \left\{
           \begin{array}{ll}
                  \displaystyle\frac{1}{2\pi} \log |x-y|
                 &\,\,\, n=2,
                 \\
                  \displaystyle\frac{1}{(2-n)c_n}|x-y|^{2-n}&\,\,\, n>2,
             \end{array}
           \right.
   \end{eqnarray*}
and $c_n$ denotes the hypersurface measure of the unit sphere in $\mathbb{R}^n$.
 For the theory of differential forms we refer the reader, for instance, to \cite{forme,flanders}.

Let
$
\Omega_0, \Omega_1, \ldots, \Omega_m
$
be $m+1$ bounded, open, connected subsets of $\mathbb{R}^n$ with $n\geq  2$, and let
$
\Sigma_j = \partial \Omega_j
$
($j=0,1,\ldots,m$) denote their (connected) boundaries, which we assume to be of class $C^{1,\lambda}$ for some $\lambda \in (0,1]$.
Moreover, we require that
$\overline{\Omega}_j \subset \Omega_0$  ($j=1,\ldots,m$),
$\overline{\Omega}_j \cap \overline{\Omega}_k = \emptyset $ ($j\neq k, j,k=1,\ldots, m$),
and define the $(m+1)$-connected domain
$
\Omega = \Omega_0 \setminus \bigcup_{j=1}^m \overline{\Omega}_j\,.
$ Obviously, 
$
\Sigma=\partial \Omega= \bigcup_{j=0}^m \Sigma_j\,.
$
If $y\in\Sigma$, with $\nu_y$ we denote  the outward
unit normal vector at  $y$.

Moreover, let  $1<p<+\infty$. As usual, in what follows   $L^{p}(\Sigma)$ and $W^{1,p}(\Sigma)$ stand
for the Lebesgue and Sobolev spaces, respectively. 

As quoted before, we are interested in looking  for 
solutions of the class of boundary value problems of type \eqref{Robin P} in terms of a double layer potential.  Specifically, we consider the following subset of real valued functions
 \begin{equation}\label{lab:spaceDp}
\mathcal{D}^p=\left\{ u:{\Omega}\to\mathbb{R}\,:\; \exists \psi \in W^{1,p}(\Sigma) \text{ s.t. } u= D\psi\right\},
\end{equation}
where
\begin{equation}\label{lab:DoublePoten}
    D\psi(x) = \int_{\Sigma} \psi(y)\,
    \frac{\partial}{\partial\nu_y} s(x,y)\, \mathrm{d}\sigma_{y},
    \qquad x\in\Omega.
\end{equation}

A particularly useful formula involving the  the double layer potential  establishes that, for any $\psi\in
W^{1,p}(\Sigma)$,
\begin{equation}\label{key formula}
  \frac{\partial }{\partial \nu_x}   \left( \int_{\Sigma} \psi(y)\frac{\partial
}{\partial \nu_y} s(x,y)\,\mathrm{d}\sigma_y\right)
  \,\mathrm{d}\sigma_x=\mathrm{d}_x\int_{\Sigma} \mathrm{d}\psi(y) \wedge s_{n-2}(x,y)
\end{equation}
on $\Sigma$ (see  \cite[p. 187]{cialdea}), where $d_x$ denotes the exterior differentiation. 
Note that this formula
provides a generalization of
formulas which usually can be found in the literature  in a different form and only for $n=2, 3$ (see, e.g.,   \cite[(1.2.14),
(1.2.15)]{HsiaoWen}).

Another type of layer potentials we use in what follows are the single layer potentials. We call 
\begin{equation}\label{lab:spaceSp}
\mathcal{S}^p=\left\{ v:\Omega\to\mathbb{R}\,:\, \exists \varphi\in L^p(\Sigma) \text{ s.t. } v=S\varphi \right\},
\end{equation}
with
\begin{equation}\label{lab:SinglePoten}
S\varphi(x)=\int_{\Sigma} \varphi(y)\,
   s(x,y)\, \mathrm{d}\sigma_{y},
    \qquad x\in\Omega.
\end{equation}

We want to point out that in $\mathbb{R}^2$ there exist simply connected domains $\Omega$ for which a non zero constant  cannot be represented by a single layer potential. When this occurs, the boundary $\Sigma$ is said to be \emph{exceptional}.
For example the unit circle is exceptional. In \cite[Theorem~3.1]{CiLeMa2012} it is shown that this phenomenon also arises for $(m+1)$-connected domains. In particular, $\Sigma$ is exceptional  if, and only if,  the exterior boundary $\Sigma_0$ is exceptional. 

Therefore, whenever $n = 2$ and $\Sigma_0$ is exceptional, we say that $v \in \mathcal{S}^p$ if, and only if, 
\begin{equation}\label{lab:SinglePotenBis}
v(x) = \int_{\Sigma} \varphi(y)\log |x-y|  \mathrm{d}s_y + c,\quad x\in\Omega,
\end{equation}
with $\varphi \in L^p(\Sigma)$ and $c \in \mathbb{R}$.

Finally,  let $B_1$ and $B_2$ be   two Banach spaces and consider a continuous linear operator $A:B_1\rightarrow B_2$.
We say that $A$ can be reduced on the
left if there exists a continuous linear
operator $A':B_2\rightarrow B_1$ such that $A'A=I+K$,
where $I$ stands for the identity operator of $B_1$ and $K:B_1\rightarrow
B_1$ is a compact operator. Analogously, one can define an operator $A$ reducible on the right.
One of the main properties of such operators is that the
equation $A\alpha=\beta$ has a solution if, and only if, $\langle
\gamma, \beta \rangle=0$ for any $\gamma$ such that $A^*\gamma=0$,
$A^*$ being the adjoint of $A$.
For further details, see for instance \cite{fichera,mikhlin}.

\section{On the Dirichlet and Neumann problems}\label{sec:DirichletNeumann}

In this section we summarize  the main results contained in \cite{CiLeMa2012}, where the Dirichlet and Neumann boundary value problems for the Laplace equation in $(m+1)$-connected domains are discussed using integral potential methods.  In particular,  the solutions are constructed via single layer potentials for the Dirichlet problem and double layer potentials for the Neumann problem,  relying on the theory of singular integral operators and on the theory of differential forms. 

The first result we mention regards the possibility to reduce on the left  the singular integral operator   $J:L^{p}(\Sigma)\longrightarrow L_{1}^{p}(\Sigma)$ defined as follows
\begin{equation}\label{lab:OperatoreJ}
J\varphi(x)= \int_{\Sigma} \varphi(y)\mathrm{d}_{x}[s(x,y)]
\mathrm{d}\sigma_y,\quad  x\in \Sigma.
\end{equation}

\begin{lemma}\label{lemma:RidLaplace}
The  singular integral operator  $J$
 can be reduced on the left by the operator  $
 J^{'}:L_{1}^{p}(\Sigma)\longrightarrow L^{p}(\Sigma)$ defined as
      \begin{equation}
            J^{'}\psi(z)=\underset{{\Sigma}}{\ast}\int_{\Sigma}\psi(x)\wedge 
            \mathrm{d}_{z}[s_{n-2}(z,x)],\quad z\in \Sigma,
                \label{lab:J'}
      \end{equation}
where  
the symbol $\underset{{\Sigma}}{\ast}$ means that, if $w$ is an $(n-1)$-form on $\Sigma$ and $w=w_0\,\mathrm{d}\sigma$, then $\underset{{\Sigma}}{\ast}w=w_0$.
\end{lemma}
 
 For the proof of Lemma \ref{lemma:RidLaplace} see \cite[Lemma 4.1]{CiLeMa2012}. Here, we only want to  remark that
 \begin{equation}\label{J'J riduzione}
J'J\varphi=-\frac{1}{4}\varphi+K^2 \varphi,
\end{equation}
where $K:L^p(\Sigma)\to L^p(\Sigma)$ is the compact linear operator
\begin{equation}\label{operatore K}
    K\varphi(x):=\int_{\Sigma} \varphi(y)\frac{\partial }{\partial \nu_x}s(x,y)\,\mathrm{d}\sigma_y,\quad \varphi\in L^p(\Sigma),\,x\in\Sigma.
\end{equation}

Thanks to Lemma \ref{lemma:RidLaplace}, it is possible to prove that, for a given $\omega\in L^p_1(\Sigma)$, there exists a solution of the singular integral equation on $\Sigma$
\begin{equation}\label{equazione integrale}
  J \varphi=\omega, \quad \varphi\in L^p(\Sigma),  
\end{equation}
if, and only if,
\begin{equation*}\label{CondCompatDirichlet}
    \int_{\Sigma}  \gamma  \wedge \omega =0
\end{equation*}
for every weakly closed form  $\gamma \in L_{n-2}^q(\Sigma)$, where
$q=(p-1)/p$ (see \cite[Theorem 4.2]{CiLeMa2012}).
Then, if   we set $\omega=\mathrm{d} f$ ($f\in W^{1,p}(\Sigma)$) in \eqref{equazione integrale},   the problem
\begin{equation*}\label{problemaLemma1}
    \left\{
      \begin{array}{ll}
      w\in \mathcal{S}^p, &\\
        \Delta w=0  & \,\, \textrm{in }\Omega, \\
        \mathrm{d}w=\mathrm{d}f & \,\,\textrm{on }\Sigma
      \end{array}
    \right.
\end{equation*}
admits a solution given by a single layer potential whose density  $\varphi\in L^p(\Sigma)$ solves the singular integral equation
\begin{equation}\label{Jphi=df}
  \int_{\Sigma} \varphi(y)\mathrm{d}_{x}[s(x,y)]
\mathrm{d}\sigma_y=\mathrm{d}f(x),\quad  x\in \Sigma,
 \end{equation}
 
Moreover, let $\mathcal{V}$ be   the eigenspace   of the Fredholm
 integral equation
\begin{equation*}
\dfrac{1}{2}\psi(x)+\displaystyle\int_{\Sigma} \psi(y)\frac{\partial}{\partial \nu_x}s(x,y)\,\mathrm{d}\sigma_y=0,\quad x\in \Sigma.
\end{equation*}
It is well-known that, if $\varphi\in L^p(\Sigma)$ belongs to $\mathcal{V}$, then $\varphi$ is H\"{o}lder continuous.
Moreover, the dimension of $\mathcal{V}$ is $m$ \cite[Proposition 3.37, p. 136]{folland}.

Then, in \cite[Lemma 4.4]{CiLeMa2012} it has been proven that, given $c_0,c_1,\ldots,c_m\in\mathbb{R}$,  the following problem
\begin{equation*} 
    \left\{
      \begin{array}{ll}
      v\in\mathcal{S}^p, &\\
        \Delta v=0 & \,\, \textrm{in }\Omega, \\
        v=c_h &\,\, \textrm{on } \Sigma_h,\,h=0,\ldots,m
      \end{array}
    \right.
\end{equation*}
has  a solution  given by
\begin{equation*} 
    v(x)=\sum_{h=1}^m (c_h-c_0)\int_{\Sigma} \Psi_h(y)s(x,y)\,d\sigma_y +c_0,\,\,\,x\in\Omega,
\end{equation*}
where $\Psi_h$ are elements of  $\mathcal{V}$ ($h=1,\ldots,m$) satisfying 
\begin{equation*}
 \int_{\Sigma} \Psi_h(y)s(x,y)\,\mathrm{d}\sigma_y=\delta_{hk},\quad \forall\, x\in\overline{\Omega}_k,\ k=1,\ldots,m.
\end{equation*}
 Taking all above results into account, the following theorem holds (see \cite[Theorem 4.5]{CiLeMa2012}.
\begin{theorem}\label{dirichlet}
The Dirichlet problem \begin{eqnarray}
       \left\{
           \begin{array}{ll}
                  u\in\mathcal{S}^{p}, & \\
                  \Delta u =0
                 & \textrm{in $\Omega$,}
                 \\
                  u=f \quad
                  & \textrm{on $\Sigma$,}
             \end{array}
           \right.
                      \label{DirichletProblem}
   \end{eqnarray}
 has a unique solution for every $f\in W^{1,p}(\Sigma)$.
\end{theorem}

In conclusion, to solve the Dirichlet problem \eqref{DirichletProblem}, we have to solve the singular integral equation \eqref{Jphi=df}, which can be reduced to a Fredholm equation by means of the operator $J'$. This reduction is not equivalent in the usual sense, since $\text{ker}(J') \neq \{0\}$  (see, e.g., \cite[p.~19]{mikhlin}). 

However,  $J'$ still yields a weaker form of equivalence. Indeed, it is shown in \cite[pp.~253--254]{cialdeamlp} that $\text{ker}(J'J) = \text{ker}(J)$. Consequently, whenever the equation $J\varphi = \psi$ admits a solution, it holds that $J\varphi = \psi$ if, and only if,  $J'J\varphi = J'\psi$. 

The form $df$ being weakly closed,   the equation $J\varphi = \mathrm{d}f$ is solvable,  and then $J\varphi = \mathrm{d}f$ if, and only if,  $\varphi$ solves the Fredholm equation $J'J\varphi = J'df$.

We pass now to consider the Neumann problem
\begin{eqnarray}
       \left\{
           \begin{array}{ll}
                  u\in\mathcal{D}^{p}, & \\
                  \Delta u =0
                 & \,\, \textrm{in $\Omega$,}
                 \\
                  \dfrac{\partial u}{\partial \nu}=g
                  &\,\,  \textrm{on $\Sigma$,}
             \end{array}
           \right.
                                \label{lab:NeumannProblem}
   \end{eqnarray}
with  the datum $g\in L^{p}(\Sigma)$ satisfying the natural compatibility condition
\begin{equation}\label{media}
\int_{\Sigma}g\, \mathrm{d}\sigma=0.
\end{equation}
   The existence (and uniqueness up
to an additive constant) theorem runs as follows.

\begin{theorem}\label{thm solvability 1}
Given $g\in L^p(\Sigma) $,  the Neumann problem (\ref{lab:NeumannProblem}) admits a solution    if, and only if,
\begin{equation}\label{media2}
    \int_{\Sigma_j}g\,\mathrm{d}\sigma=0,\quad j=0,1,\ldots,m.
\end{equation}
The solution is uniquely determined up to an additive constant.\\ Moreover,  the double layer potential (\ref{lab:DoublePoten}) is a solution of (\ref{lab:NeumannProblem}) if, and only if, its density $\psi$  is given by
\begin{equation*} 
    \psi(x)=\int_{\Sigma} \varphi(y)s(x,y)\,\mathrm{d}\sigma_y,\quad x\in\Sigma,
\end{equation*}
  where $\varphi$ is a solution of the Fredholm equation
  \begin{equation*} 
  -\frac{1}{4}\varphi+K^2\varphi=g,
  \end{equation*}
  $K$ being the operator (\ref{operatore K}).
\end{theorem}

Among other things, Theorem \ref{thm solvability 1}  shows that we can represent a solution of the Neumann problem (\ref{lab:NeumannProblem}) by means of  a double layer potential if, and only if, conditions (\ref{media2}) are satisfied.

If $g$ satisfies  the only condition   (\ref{media}),
  we need to modify the representation of the solution by adding some extra terms, as stated in \cite[Theorem 5.4]{CiLeMa2012}.   
\begin{theorem} 
Given $g\in L^p(\Sigma) $   satisfying  (\ref{media}), consider the Neumann problem
\begin{eqnarray}
       \left\{
           \begin{array}{ll}
                  u\in\widetilde{\mathcal{S}}^{p}, & \\
                  \Delta u =0
                 & \,\,\textrm{in $\Omega$,}
                 \\
                  \dfrac{\partial u}{\partial \nu}=g
                  &\,\, \textrm{on $\Sigma$}\,,
             \end{array}
           \right.
                                \label{NeumannProblemBis}
   \end{eqnarray}
   where  the symbol  $\widetilde{\mathcal{S}}^p$ stands for  the space  of all harmonic functions 
\begin{equation}\label{sol gen}
      u(x)=\int_{\Sigma} \psi(y)\frac{\partial}{\partial \nu_y}s(x,y)\,\mathrm{d}\sigma_y+\sum_{j=1}^m  c_j \int_{\Sigma_j} s(x,y)\,\mathrm{d}\sigma_y, \quad x\in\Omega,
\end{equation}
where $\psi\in W^{1,p}(\Sigma)$ and $ c_j\in\mathbb{R} $ ($j=1,\ldots,m$). 
 
 Then, problem  \eqref{NeumannProblemBis} admits a solution as in  \eqref{sol gen} with 
\begin{equation*} 
  c_j= - \dfrac{1}{|\Sigma_j|}\int_{\Sigma_j} g\mathrm{d}\sigma\quad (j=1,\ldots, m)\,.
\end{equation*}
The solution is uniquely determined up to an additive constant.
\end{theorem}

\section{On the  Robin problem}\label{sec:Robin}

 Let us consider the Robin problem  
\begin{eqnarray}
       \left\{
           \begin{array}{ll}
                  \Delta u =0
                 & \,\, \textrm{in $\Omega$,}
                 \\
                 \dfrac{\partial u}{\partial \nu}+hu=g
                  &\,\,  \textrm{on $\Sigma$,}
             \end{array}
           \right.
                                \label{lab:RobinProblem}
   \end{eqnarray}
where the datum  $g$ is given in $L^{p}(\Sigma)$, $1<p<\infty$.
 
Note that if $h= 0$ a.e. on $\Sigma$, then the  problem \eqref{lab:RobinProblem} turns into the Neumann problem.  

We assume that $h\in L^\infty(\Sigma)$ is such that
 \begin{equation}\label{funzione h}
 h\geq 0 \,\, \text{ and } \,\,\int_\Sigma h\, \mathrm{d}\sigma>0.
\end{equation}

  In  \cite[Theorem 5.11.5]{medkova} it  have been shown that,   if $\Omega$ is a bounded domain with a compact boundary of class $C^{1,\lambda}$, with $0<\lambda<1$,  and if  condition \eqref{funzione h} holds, then  the problem \eqref{lab:NeumannProblem} admits a unique solution  which can be represented by means of a single layer potential.
  
Now we are interested to solve the problem
\begin{eqnarray}  \label{lab:RobinProblemBis}
       \left\{
           \begin{array}{ll}
                 u\in\mathcal{D}^{p}, & \\
                  \Delta u =0
                 & \,\, \textrm{in $\Omega$,}
                 \\
                 \dfrac{\partial u}{\partial \nu}+hu=g
                  &\,\,  \textrm{on $\Sigma$,}
             \end{array}
           \right.
     \end{eqnarray}
in the space $\mathcal{D}^p$ defined by \eqref{lab:spaceDp}, where $h$ satisfies the hypothesis \eqref{funzione h}  and $\Omega$ is an ($m+1$)-connected domain as defined in Section \ref{sec:intro}. 

Let $u=D\psi$ be a double layer potential   with density a function $\psi\in W^{1,p}(\Sigma)$  (see definition \eqref{lab:DoublePoten}).
On account of  \eqref{key formula},   the boundary condition $ {\partial u}/{\partial \nu}+hu=g$  can be rewritten on $\Sigma$  as  
\begin{equation*}\label{boundarycondition}
    \int_{\Sigma} \mathrm{d}\psi(y)\wedge \mathrm{d}_x[s_{n-2}(x,y)]+h \left( \frac{1}{2}\psi +D\psi \right)d\sigma =g\,\mathrm{d}\sigma,
\end{equation*}
where we have used the jump relation for the double layer potential $D\psi$. 

Recalling the definition  \eqref{lab:J'} of the operator $J'$, the above equality becomes   
\begin{equation}\label{integral equation}
    J'(\mathrm{d}\psi)+h \left( \frac{1}{2}\psi +D\psi\right)=g\quad \text{ on }\Sigma\,.
\end{equation}
A solution of the problem (\ref{lab:RobinProblem}) exists  if, and only if, the equation \eqref{integral equation}  is solvable.

Assume either $n\geq 3$, or that $n=2$ with  $\Sigma$ not exceptional. Let us write  $\psi\in W^{1,p}(\Sigma)$   as a single layer potential $S\varphi$  with $\varphi\in L^p(\Sigma)$ (see \eqref{lab:SinglePoten}).
Hence, $\mathrm{d}\psi=J\varphi$, where $J$ is given by \eqref{lab:OperatoreJ}, and the first term in \eqref{integral equation} can be rewritten as $J'(\mathrm{d}\psi)=J'J\varphi$.  Therefore,  using \eqref{J'J riduzione}, equation \eqref{integral equation} is equal to the following integral equation of Fredholm type
\begin{equation}\label{integral equation 2}
   -\frac{1}{4} \varphi +K^2 \varphi+h \left( \frac{1}{2}S\varphi +D(S\varphi)\right)=g\quad \text{ on }\Sigma\,,
\end{equation}
where we recall that $K$ is the compact operator defined in \eqref{operatore K}. 
This shows that the operator  on the left hand-side of   \eqref{integral equation}, mapping from $W^{1,p}(\Sigma)$ into $L^p(\Sigma)$, can be reduced on the right.  Then, there exists a solution of \eqref{integral equation} if, and only if,  $g$ satisfies the relevant compatibility conditions.   However, in this case, we can prove the existence of a solution to  \eqref{integral equation} in a way that is probably easier: 
due to Theorem \ref{dirichlet}, any function $\psi\in W^{1,p}(\Sigma)$ can be written as a single layer potential $S\varphi$  with $\varphi\in L^p(\Sigma)$.
Therefore, a solution  $\psi\in W^{1,p}(\Sigma)$ to \eqref{integral equation} exists if, and only if,  a solution $\varphi\in L^p(\Sigma)$ to the equation \eqref{integral equation 2} exists. 
As a consequence,   the Robin problem (\ref{lab:RobinProblemBis}) admits a solution  if, and only if,  the Fredholm equation  \eqref{integral equation 2}  is solvable.
 If we set
\begin{equation}\label{operatore H}
H \varphi=   K^2\varphi+h \left( \frac{1}{2}S\varphi +D(S\varphi) \right)\quad \text{ on }\Sigma\,,
\end{equation}
this happens for any $g\in L^p(\Sigma)$  if, and only if,  $\text{ker}  \left(-\frac{1}{4} I+H \right)=\{0\}$. Hence, let $\varphi\in \text{ker}   \left(-\frac{1}{4} I+H \right)$, that is 
\[
-\frac{1}{4} \varphi + H\varphi=0\,.
\]
By bootstrap techniques, $\varphi$ is a H\"older function
on $\Sigma$ and, as a consequence,  $\psi=S\varphi
\in C^{1,\lambda}(\Sigma)$. This implies  that $ \psi\in W^{1,2}(\Sigma)$, and hence $u=D\psi\in \mathcal{D}^2$. In this case,  as proved in \cite[Remark 1]{CiLeMa2012},  the usual formula
\begin{equation*}\label{lab_Green3}
   \int_\Omega |\nabla u|^2 \, \mathrm{d}x=\int_\Sigma u\frac{\partial u}{\partial \nu}\, \mathrm{d}\sigma
\end{equation*}
holds true, where ${\partial u}/{\partial \nu}$ is given by (\ref{key formula}) . Therefore,
\[
0=\int_\Sigma u\left( \dfrac{\partial u}{\partial \nu}+hu\right)\mathrm{d}\sigma=\int_\Omega|\nabla u|^2 \mathrm{d}x+\int_\Sigma hu^2 \mathrm{d}\sigma\,,
\]
 which implies that $|\nabla u|^2=0$ in $\Omega$ and $hu^2=0$ on $\Sigma$. Then, there exists a constant $c$ such that $u= c$ in $\Omega$ and, from  equality $\int_{\Sigma} hc^2\,\mathrm{d}\sigma=0$, 
 it follows that necessarily $u= 0$, and then $\varphi=0$. 

 Accordingly, we have proven the following result of existence and uniqueness of the solution of the Robin problem \eqref{lab:RobinProblemBis}.

 \begin{theorem} \label{teorema 4}
Assume either $n\geq 3$, or that $n=2$ with  $\Sigma$ not exceptional.  Then, the Robin problem (\ref{lab:RobinProblemBis}) admits a unique solution   for every $g\in L^{p}(\Sigma)$. Moreover,  the double layer potential (\ref{lab:DoublePoten}) is a solution of (\ref{lab:RobinProblemBis}) if, and only if, its density $\psi$  is given by
\begin{equation*}\label{psi semplice strato}
    \psi(x)=\int_{\Sigma} \varphi(y)s(x,y)\,\mathrm{d}\sigma_y,\quad x\in\Sigma,
\end{equation*}
  where $\varphi$ is the solution of the Fredholm equation 
  \begin{equation*}\label{eqphi}
  -\frac{1}{4}\varphi+H\varphi=g\quad \text{ on } \Sigma\,,
  \end{equation*}
  $H$ being the operator (\ref{operatore H}).
\end{theorem}

We have also
\begin{theorem}
    Suppose $n=2$ and $\Sigma$ is exceptional. Then,   the Robin problem (\ref{lab:RobinProblemBis}) admits a unique solution   for every $g\in L^{p}(\Sigma)$. Moreover,  the double layer potential (\ref{lab:DoublePoten}) is a solution of (\ref{lab:RobinProblemBis}) if, and only if, its density $\psi$  is given by
    \begin{equation*}\label{psi c semplice strato}
    \psi(x)=\int_{\Sigma} \varphi(y)s(x,y)\,\mathrm{d}\sigma_y + c,\quad x\in\Sigma,
\end{equation*}
where $(\varphi,c)\in L^p(\Sigma)\times\mathbb{R}$ is the solution of the Fredholm equation 
 \begin{equation}\label{eqphic}
  -\frac{1}{4}(\varphi,c)+\widetilde{H}(\varphi,c)=(g,0)\quad \text{ on } \Sigma\,,
  \end{equation}
  and
  \begin{equation}\label{eq:defHphic}
\widetilde{H}(\varphi,c)= \left(H\varphi+hc,  \int_{\Sigma}\varphi\, ds + \frac{c}{4}\right).
\end{equation}
\end{theorem}
\begin{proof}
The proof is similar to the previos one. As before, a solution of the problem (\ref{lab:RobinProblemBis}) exists  if, and only if, the equation \eqref{integral equation}  is solvable. Now, we cannot write a function $\psi\in W^{1,p}(\Sigma)$ as
a single layer potential $S\varphi$, but we have to write
\begin{equation}\label{eq:modified}
\psi = S\varphi + c
\end{equation}
(see \eqref{lab:SinglePotenBis} and Theorem \ref{dirichlet}).
We note that, if we require
\begin{equation}\label{eq:condort0}
\int_{\Sigma} \varphi\, ds = 0\,,
\end{equation}
then $(\varphi,c)$ is uniquely determined. In fact, suppose that
$$
S\varphi_1 + c_1 = S\varphi_2 + c_2 \quad \text{on } \Sigma.
$$
This means $S(\varphi_1-\varphi_2)=c_1-c_2$ on $\Sigma$ and, since $\Sigma$ is exceptional, we must have
$c_1-c_2=0$, and then
$$
S(\varphi_1-\varphi_2)=0 \quad \text{on } \Sigma.
$$
This leads to $S(\varphi_1-\varphi_2)=0$ in $\Omega$ and then, thanks to condition \eqref{eq:condort0}, 
$S(\varphi_1-\varphi_2)=0$ also in $\mathbb{R}^{2}\setminus \overline{\Omega}$. By using the jump relations for
the normal derivative of a single layer potential, we find $\varphi_1-\varphi_2=0$ on $\Sigma$.

Coming back to equation  \eqref{integral equation},  since we have $\mathrm{d}\psi= \mathrm{d}(S\varphi + c)=\mathrm{d} (S\varphi)$, it becomes

$$
   -\frac{1}{4} \varphi +K^2\varphi+h \left( \frac{1}{2}S\varphi +D(S\varphi) \right) + h\left(
   \frac{c}{2} + Dc\right) =g \quad \text{ on }\Sigma\,,
$$
i.e.
\begin{equation}\label{integral equation 3}
   -\frac{1}{4} \varphi + H\varphi+ h\, c =g \quad \text{ on }\Sigma.
\end{equation}

We now prove that equation \eqref{integral equation 3}  has one and only one solution $(\varphi,c)\in L^p(\Sigma)\times \mathbb{R}$, with  
condition \eqref{eq:condort0}.  Let us consider the Fredholm equation
\begin{equation}\label{eq:fredh2}
- \frac{1}{4}(\varphi,c) + \widetilde{H}(\varphi,c) = (g,0)
\end{equation}
and let us prove that its kernel is trivial.  We note that $(\varphi,c)$ is an eigensolution of this equation if, and only if, 
\begin{equation}\label{eq:system}
\begin{cases}
-\dfrac{1}{4} \varphi + H\varphi+ h\, c =0 ,  \\[0.2cm]
\displaystyle\int_{\Sigma} \varphi\, ds =0\, .
\end{cases}
\end{equation}

From the first equation,  setting $u=D(S\varphi+c)= D(S\varphi) + c$ in $\Omega$ and, reasoning as in the proof of Theorem \ref{teorema 4},
we find $u=0$ in $\Omega$. This means $D(S\varphi + c) =0$ and then $S(\varphi) + c=0$ on $\Sigma$. But $\Sigma$
being exceptional, we must have $c=0$. The second equation in \eqref{eq:system} implies $\varphi=0$.
Therefore,  Fredholm equation \eqref{eq:fredh2} has one and only one solution for any $g\in L^p(\Sigma)$. But this is equivalent
to say that  equation \eqref{integral equation 3} is satisfied and \eqref{eq:condort0} holds. 
This completes the proof of the theorem.
\end{proof}

\begin{acknowledgement}
 The authors are members of the “Gruppo Nazionale per l’Analisi Matematica, la Probabilità e le loro Applicazioni” (GNAMPA) of the “Istituto Nazionale di Alta Matematica” (INdAM).
The authors acknowledge the support of the Project funded by the European Union - Next Generation EU under the National Recovery and Resilience Plan (NRRP), Mission 4 Component 2 Investment 1.1 - Call for tender PRIN 2022 No. 104 of February, 2 2022 of Italian Ministry of University and Research; Project code: 2022SENJZ3 (subject area: PE - Physical Sciences and Engineering) “Perturbation problems and asymptotics for elliptic differential equations: variational and potential theoretic methods”.
\end{acknowledgement}
\ethics{Competing Interests}{\newline
This study was funded by European Union - Next Generation EU under the  the National Recovery and Resilience Plan (NRRP), Mission 4 Component 2 Investment 1.1. Project number 2022SENJZ3.\newline
The authors have no conflicts of interest to declare that are relevant to the content of this chapter.}

\input{references}

\end{document}

%% file: references.tex
%
%
%

%% file: CialdeaLeonessa.bbl
\begin{thebibliography}{99.}%
%
%

\bibitem{cialdea} Cialdea A.: On oblique derivate problem for Laplace equation and connected topics. 
Rend. Accad. Naz. Sci. XL Mem. Mat. (5) \textbf{12},  no. 1, 181--200 (1988)

 
\bibitem{cialdeamlp} Cialdea A. The multiple layer potential for the biharmonic equation in $n$ variables.
Atti Accad. Naz. Lincei Cl. Sci. Fis. Mat. Natur. Rend.Lincei (9) Mat. Appl. \textbf{3} , no. 4, 241--259 (1992) 


\bibitem{ACCialdeanew}  Cialdea  A.:  The Simple-Layer Potential Approach to the Dirichlet Problem: An Extension to Higher Dimensions of Muskhelishvili Method and Applications.  In:   C. Constanda,  M. Dalla Riva, P.D. Lamberti, P. Musolino (Eds.)  Integral Methods in Science and Engineering, Volume 1, pp. 59--69 , Birkh\"auser
(2017)
 
\bibitem{CialdeaHsiao1995}  Cialdea A.,   Hsiao G.C.:  
Regularization for some boundary integral equations of the first kind in Mechanics.
Rend. Accad. Naz. Sci. XL, Serie 5, Mem. Mat. Parte I.  \textbf{XIX},   25–42  (1995)

\bibitem{CiLeMa2012}  Cialdea A.,  Leonessa V.,  Malaspina A.:  On the Dirichlet and the Neumann problems for Laplace equation in multiply connected domains.
 Complex Var. Elliptic Equ. \textbf{57}, no. 10, 1035-- 1054 (2012)  
 
 
  \bibitem{Cialema2011}  Cialdea A.,  Leonessa V.,  Malaspina A.:   Integral representations for solutions of some BVPs for the Lamé system in
multiply connected domains. Bound. Value Probl. \textbf{2011}, Article ID: 53 (2011)

  \bibitem{ACCiLeMa4}   Cialdea A.,   Leonessa V.,   Malaspina A.:  The Dirichlet problem for second-order divergence form elliptic operators with variable coefficients: the
simple layer potential ansatz. 
 Abstr. Appl. Anal. \textbf{2015}, Art. ID 276810, 11 pages (2015)

\bibitem{cialema2023} 
Cialdea A.,   Leonessa V.,   Malaspina A.:  On the traction problem for steady elastic oscillations equations: the double
layer potential ansatz. Rend. Circ. Mat. Palermo, II. Ser., \textbf{72}, no. 3, 1947--1960  (2023)

\bibitem{fichera}  Fichera G.: Una introduzione alla teoria delle equazioni integrali singolari.
   Rend. Mat. Roma 5 \textbf{17}, 82--191 (1958)


\bibitem{forme}  Fichera G.:  Spazi lineari di $k$-misure e di forme differenziali.   Proc. of Intern. Symposium on Linear Spaces,
Jerusalem 1960, Israel Ac. of Sciences and Humanities, Pergamon
Press,  175--226 (1961)



 
\bibitem{flanders}  Flanders H. : Differential Forms with Applications to the Physical Sciences.  Academic Press, New York, San Francisco, London (1963)

\bibitem{folland}  Folland G.B.:  An introduction to partial differential equations.   Princeton University Press (1995)

\bibitem{hodge}  Hodge W.V.:  A Dirichlet problem for harmonic functionals with applications to analytic varieties.  Proc. of the London
        Math. Soc. 2 \textbf{36},  257--303 (1934) 
        
\bibitem{HsiaoWen}  Hsiao G.C. ,  Wendland W.L.:  Boundary integral equations.   AMS 164, Springer (2008)

\bibitem{medkova} Medkov\'a D.:  The Laplace equation. Boundary value problems on bounded and unbounded Lipschitz domains.
Springer, Cham (2018)

\bibitem{mikhlin}  Mikhlin S.G., Pr\"ossdorf  S.:    Singular integral operators. Springer-Verlag, Berlin (1986)





 \end{thebibliography}
